\documentclass[11pt]{amsart}

\usepackage{amsmath, amssymb, amsthm, mathrsfs, graphicx}
\usepackage{hyperref, enumitem, xspace, ifthen,comment}
\usepackage[all]{xy}
\usepackage{tikz}
\usetikzlibrary{arrows,shapes,trees}

\newcommand{\mypagesize}{
\textwidth= 6.25in
\textheight=8.75in
\voffset-.5in
\hoffset-.75in
\marginparwidth=56pt
}
\mypagesize

\newtheorem*{thm-plain}{Theorem}
\newtheorem{thm}{Theorem}[section]
\newtheorem{lem}[thm]{Lemma}
\newtheorem{prp}[thm]{Proposition}
\newtheorem{cor}[thm]{Corollary}

\newtheorem{conj}[thm]{Conjecture}
\newtheorem{ques}[thm]{Question}
\numberwithin{equation}{thm}

\theoremstyle{definition}
\newtheorem{dfn}[thm]{Definition}

\newtheorem*{dfn-plain}{Definition}

\theoremstyle{remark}
\newtheorem{clm}[thm]{Claim}

\newtheorem{awlog}[thm]{Additional Assumption}

\newtheorem{rem}[thm]{Remark}

\newtheorem{exm}[thm]{Example}
\newtheorem*{rem-plain}{Remark}

\DeclareMathOperator{\Spec}{Spec}

\DeclareMathOperator{\Pic}{Pic}
\DeclareMathOperator{\codim}{codim}
\DeclareMathOperator{\tor}{tor}

\DeclareMathOperator{\Exc}{Exc}

\DeclareMathOperator{\Cl}{Cl}
\DeclareMathOperator{\Cln}{Cl_{num}}

\def\rd#1.{\lfloor{#1}\rfloor}
\def\rp#1.{\lceil{#1}\rceil}

\newcommand{\lto}{\longrightarrow}

\newcommand{\N}{\mathbb N}
\newcommand{\Z}{\mathbb Z}
\newcommand{\Q}{\ensuremath{\mathbb Q}}
\newcommand{\R}{\mathbb R}
\newcommand{\C}{\mathbb C}

\newcommand{\A}{\mathbb A}
\renewcommand{\P}{\mathbb P}
\renewcommand{\O}{\mathscr O}

\newcommand{\x}{\times}

\renewcommand{\phi}{\varphi}
\renewcommand{\theta}{\vartheta}

\newcommand{\minus}{\setminus}

\newcommand{\surj}{\twoheadrightarrow}

\newcommand{\isom}{\cong}
\newcommand{\mf}{\mathfrak}

\newcommand{\tensor}{\otimes}

\DeclareMathOperator{\pr}{pr}

\newcommand{\tors}{\mathrm{tor}}
\newcommand{\ntl}{\natural}
\newcommand{\Gr}{\mathrm{Gr}}

\newcommand{\sE}{\mathscr{E}}
\newcommand{\sF}{\mathscr{F}}

\newcommand{\cI}{\mathcal{I}}

\newcommand{\cJ}{\mathcal{J}}

\newcommand{\cK}{\mathcal{K}}
\newcommand{\sL}{\mathscr{L}}

\newcommand{\sO}{\mathscr{O}}

\newcommand{\cQ}{\mathcal{Q}}

\newdir{ ir}{{}*!/-5pt/@^{(}} 
\newdir{ il}{{}*!/-5pt/@_{(}} 

\DeclareMathOperator{\dv}{div}
\DeclareMathOperator{\supp}{supp}

\newcommand{\eps}{\varepsilon}
\newcommand{\conv}{\rightarrow}

\newcommand{\wt}{\widetilde}

\newcommand{\wb}{\overline}

\newenvironment{sequation}{%
\setcounter{equation}{\value{thm}}%
\numberwithin{equation}{section}%
\begin{equation}%
}{%
\end{equation}%
\numberwithin{equation}{thm}%
\addtocounter{thm}{1}%
}

\numberwithin{equation}{thm}

\DeclareRobustCommand{\SkipTocEntry}[5]{}

\setitemize[1]{leftmargin=*,parsep=0em,itemsep=0.125em,topsep=0.125em}
\setenumerate[1]{leftmargin=*,parsep=0em,itemsep=0.125em,topsep=0.125em}

\newcommand{\iref}[3]{\the\value{#1}.\the\value{#2}(\the\value{#3})}

\newcommand\factor[2]{\left. \raise 2pt\hbox{$#1$} \right/\hskip -2pt \raise -2pt\hbox{$#2$}}

\definecolor{forrest}{RGB}{81,133,49}
\definecolor{mydarkblue}{RGB}{10,92,153}

\begin{document}

\title{The jumping coefficients of non-\Q-Gorenstein multiplier ideals}
\author{Patrick Graf} %
\address{Lehrstuhl f\"ur Mathematik I, Universit\"at Bay\-reuth,
  95440 Bayreuth, Germany} %
\email{\href{mailto:patrick.graf@uni-bayreuth.de}{patrick.graf@uni-bayreuth.de}}
\date{\today}
\thanks{The author was supported in full by a research grant of the
  Deutsche Forschungs\-gemeinschaft (DFG)}  %

\begin{abstract}
Let $\mf a \subset \O_X$ be a coherent ideal sheaf on a normal complex variety $X$, and let $c \ge 0$ be a real number. De Fernex and Hacon associated a multiplier ideal sheaf to the pair $(X, \mf a^c)$ which coincides with the usual notion whenever the canonical divisor $K_X$ is $\Q$-Cartier. We investigate the properties of the jumping numbers associated to these multiplier ideals.

We show that the set of jumping numbers of a pair is unbounded, countable and satisfies a certain periodicity property.
We then prove that the jumping numbers form a discrete set of real numbers if the locus where $K_X$ fails to be $\Q$-Cartier is zero-dimensional.
It follows that discreteness holds whenever $X$ is a threefold with rational singularities.

Furthermore, we show that the jumping numbers are rational and discrete if one removes from $X$ a closed subset $W \subset X$ of codimension at least three, which does not depend on $\mf a$.
We also obtain that outside of $W$, the multiplier ideal reduces to the test ideal modulo sufficiently large primes $p \gg 0$.
\end{abstract}

\maketitle


\section{Introduction}

The theory of multiplier ideal sheaves has become an important part of complex algebraic and complex analytic geometry. To a coherent ideal sheaf $\mf a \subset \O_X$ on a smooth (or more generally normal and \Q-Gorenstein) complex variety $X$ and a real exponent $c \ge 0$, it associates an ideal sheaf $\cJ(X, \mf a^c)$ satisfying strong vanishing theorems.
If the subscheme $Z \subset X$ corresponding to $\mf a$ is a divisor $D$, the multiplier ideal can be thought of as measuring the failure of the pair $(X, cD)$ to be klt.

In particular, multiplier ideals are able to detect the \emph{log canonical threshold} of $D$, which is defined as the smallest number $t$ such that $(X, tD)$ is log canonical but not klt, and which is an important invariant of the pair $(X, D)$. However, using multiplier ideals one sees that the log canonical threshold is merely the first member of an infinite sequence of numbers, the \emph{jumping coefficients} (or jumping numbers) attached to $(X, D)$. Intuitively, the multiplier ideals $\cJ(X, tD)$ get smaller as $t$ increases, and the jumping coefficients are those values of $t$ where $\cJ(X, tD)$ jumps.
These numbers first appeared implicitly in the work of Libgober~\cite{Lib83} and Loeser and Vaqui\'e~\cite{LV90}. They were studied systematically by Ein, Lazarsfeld, Smith, and Varolin~\cite{ELSV04}.

On the other hand, in positive characteristic there is the theory of test ideals, which is understood to be the analogue of the theory of multiplier ideals (see e.g.~\cite{Smi00}). Test ideals can be defined for any variety, that is, no \Q-Gorenstein assumption is required. It is thus natural to wonder whether also in characteristic zero, the theory of multiplier ideals can be extended to arbitrary (say normal) varieties. Such an extension was developed by de Fernex and Hacon~\cite{dFH09} and elaborated on by Boucksom, de Fernex, and Favre~\cite{BdFF12}. Multiplier ideals in this generality are still poorly understood, partly due to the asymptotic nature of their definition involving infinitely many resolutions of singularities.

The aim of this paper is to study jumping numbers in the non-\Q-Gorenstein case. Let us fix our definitions.

\begin{dfn}[Pairs and jumping numbers] \label{dfn:pairs jumping numbers intro}
A \emph{pair} $(X, Z)$ consists of a normal complex variety $X$ and a proper closed subscheme $Z \subsetneq X$. A positive real number $\xi > 0$ is called a \emph{jumping number} of the pair $(X, Z)$ if
\[ \cJ(X, \xi Z) \subsetneq \cJ(X, \lambda Z) \quad \text{for all $0 \le \lambda < \xi$}. \]
The set of jumping numbers of $(X, Z)$ is denoted by $\Xi(X, Z) \subset \R^+$.
\end{dfn}

For the reader's convenience, in Sections~\ref{sec:Weil pullback} and~\ref{sec:non-Q-Gor mult ideals} of this paper we recall the definition of the multiplier ideals $\cJ(X, cZ)$ according to~\cite{dFH09}.
We then begin by establishing some basic properties of jumping numbers, which give a first idea what the set of jumping numbers of a pair looks like.

\begin{prp}[Basic properties of jumping numbers] \label{prp:elem}
Let $(X, Z)$ be a pair.
\begin{enumerate}
\item\label{itm:elem.emp} (Nonemptiness) If $Z \ne \emptyset$, then $\Xi(X, Z) \ne \emptyset$.
\item\label{itm:elem.unb} (Unboundedness) If $\xi \in \Xi(X, Z)$, then also $\xi + 1 \in \Xi(X, Z)$. In particular, if $Z \ne \emptyset$ then the set $\Xi(X, Z)$ is unbounded above.
\item\label{itm:elem.dcc} (DCC property) The set $\Xi(X, Z)$ satisfies the descending chain condition, i.e.~any decreasing subsequence of $\Xi(X, Z)$ becomes stationary. In particular, if $\xi \ge 0$ is any real number, then $(\xi, \xi + \eps] \cap \Xi(X, Z) = \emptyset$ for sufficiently small $\eps > 0$, depending on $\xi$.
\item\label{itm:elem.abz} (Countability) The set $\Xi(X, Z)$ is countable.
\item\label{itm:elem.per} (Periodicity) If $\xi > \dim X - 1$, then $\xi \in \Xi(X, Z)$ if and only if $\xi + 1 \in \Xi(X, Z)$.
\end{enumerate}
\end{prp}

In the \Q-Gorenstein case, it is elementary to see that for any pair $(X, Z)$ the jumping numbers $\Xi(X, Z)$ form a discrete set of rational numbers. De Fernex and Hacon~\cite[Rem.~4.10]{dFH09} asked whether this still holds true in general.

\begin{ques} \label{ques:dFH}
Let $(X, Z)$ be a pair. Then is the set $\Xi(X, Z)$ of jumping numbers a discrete set of rational numbers?
\end{ques}

A positive answer at least to the discreteness part of the question is known in several cases: if $X$ is projective with at most log terminal\footnote{in the sense of~\cite{dFH09}, i.e.~without assuming that $K_X$ is \Q-Cartier} or isolated singularities~\cite[Thm.~5.2]{Urb12}, if $X$ is a toric variety~\cite[Sec.~5]{Urb12b}, or if $K_X$ is numerically \Q-Cartier~\cite[Thm.~1.3]{BdFFU13}. For the definition of numerically \Q-Cartier divisors and numerically \Q-factorial varieties, see Section~\ref{sec:num Q-Cartier} below.

While it is expected that the set $\Xi(X, Z)$ is always discrete, at the moment we are unable to prove this. Therefore we have to content ourselves with considering special classes of singularities. Our two main results are the following.

\begin{thm}[Discreteness for isolated non-\Q-Gorenstein loci] \label{thm:discrete isol}
Let $(X, Z)$ be a pair such that the non-\Q-Cartier locus of $K_X$ is zero-dimensional. Then $\Xi(X, Z)$ is a discrete subset of $\R$.
\end{thm}

Recall that the non-\Q-Cartier locus of a Weil divisor $D$ on a normal variety $X$ is defined as the closed subset of $X$ consisting of those points where all positive multiples $mD$ fail to be Cartier.

\begin{thm}[Discreteness in dimension three] \label{thm:discrete 3}
Let $(X, Z)$ be a pair, where $X$ is a normal threefold (not necessarily projective) whose locus of non-rational singularities is zero-dimensional. Then $\Xi(X, Z)$ is discrete.
\end{thm}

\begin{rem}
By Proposition~\ref{prp:elem}.\ref{itm:elem.dcc}, the conclusion about discreteness can be rephrased as saying that $\Xi(X, Z)$ satisfies ACC (the ascending chain condition) for bounded subsequences.
\end{rem}

Theorem~\ref{thm:discrete 3} follows from Theorem~\ref{thm:discrete isol} combined with the following result also proved in this article.

\begin{thm}[Generic numerical \Q-factoriality] \label{thm:gen num Q-fact}
Let $X$ be a normal complex variety. Then there is a closed subset $W \subset X$ of codimension at least three such that $X \minus W$ is numerically \Q-factorial.
\end{thm}

Concerning rationality of jumping numbers, Urbinati~\cite[Thm.~3.6]{Urb12} gave an example of a normal projective threefold $X$ with a single isolated singularity $x \in X$ such that for the reduced subscheme $Z = \{ x \}$, the set $\Xi(X, Z)$ consists only of irrational numbers. From Theorem~\ref{thm:gen num Q-fact} and~\cite[Thm.~1.3]{BdFFU13}, we see that quite generally this phenomenon can only happen in codimension at least three.

\begin{cor}[Generic discreteness and rationality] \label{cor:gen discrete}
Let $(X, Z)$ be a pair. Then there is a dense open subset $U \subset X$, not depending on $Z$, whose complement has codimension at least three and such that $\Xi(U, Z|_U)$ is a discrete set of rational numbers.
\end{cor}

\begin{rem}
In the example of Urbinati mentioned above, all the jumping numbers are algebraic. What is more, all of them are contained in $\Q(\sqrt{17})$. In view of this, it is natural to ask whether jumping numbers can also be transcendental. It certainly seems reasonable to expect this. However, jumping numbers are hard to compute in concrete examples, especially in the non-\Q-Gorenstein case.
\end{rem}

Under a \Q-Gorenstein assumption, multiplier ideals are known to reduce to the corresponding test ideals in sufficiently large characteristic~\cite{Smi00, Har01, HY03, Tak04}. It has been asked whether this is still true in general (see e.g.~\cite[Rem.~6.2]{Sch11}). An affirmative partial answer was given in~\cite[Thm.~1]{dFDTT14}. Using that result, from Theorem~\ref{thm:gen num Q-fact} we deduce the following statement.

\begin{cor}[Generic comparison to test ideals] \label{cor:gen test}
Let $(X, Z)$ be a pair, where $Z$ is an effective \Q-Cartier Weil divisor on $X$. Then there is a dense open subset $U \subset X$ independent of $Z$, whose complement has codimension at least three, such that the following holds:
Given a model of $(U, Z|_U)$ over a finitely generated $\Z$-subalgebra $A$ of $\C$, and a rational number $c \ge 0$, there is a dense open subset $S \subset \Spec A$ such that for all closed points $s \in S$, we have
\[ \cJ(U, cZ|_U)_s = \tau(U_s, c(Z|_U)_s). \]
Here $\tau(U_s, c(Z|_U)_s)$ denotes the big (non-finitistic) test ideal of the pair $(U_s, c(Z|_U)_s)$.
\end{cor}

\subsection*{Outline of proof of Theorem~\ref{thm:discrete isol}}

The key point is the following conjecture of Urbinati \cite[Conj.~4.6]{Urb12b}.

\begin{conj}[Global generation conjecture] \label{conj:glob gen}
Let $X$ be a complex normal projective variety and $D$ a Weil divisor on $X$. Then there is an ample Cartier divisor $H$ on $X$ such that for any $m \in \N$, the sheaf $\O_X(m(D + H))$ is globally generated.
\end{conj}

It was noted by Urbinati in~\cite[Sec.~5]{Urb12} that Conjecture~\ref{conj:glob gen} is closely related to the question of discreteness of jumping numbers. In fact, a proof of the conjecture would provide an unconditionally positive answer to that question. We establish a weak form of the conjecture that only deals with isolated points of the non-\Q-Cartier locus of the Weil divisor $D$ in question~(Theorem~\ref{thm:glob gen isol}).
Turning to the proof of Theorem~\ref{thm:discrete isol}, assume first for simplicity that $X$ is projective, and suppose by way of contradiction that we are given a strictly descending chain of multiplier ideals
\begin{sequation} \label{eqn:chain}
\cJ(X, t_1 Z) \supsetneq \cJ(X, t_2 Z) \supsetneq \cdots,
\end{sequation}%
where the sequence $(t_k)$ converges to an accumulation point $t_0$ of $\Xi(X, Z)$. Theorem~\ref{thm:glob gen isol} enables us to find an ample line bundle $\sL$ on $X$ such that $\sL \tensor \cJ(X, t_k Z)$ is globally generated for all $k \ge 1$. Taking global sections in~\eqref{eqn:chain} then yields a contradiction.

In general, if $X$ is a quasi-projective variety such that the non-\Q-Cartier locus of $K_X$ is zero-dimensional, to apply the above reasoning we need to compactify $X$ to a projective variety $\wb X$. The non-\Q-Cartier locus of $K_{\wb X}$ may then no longer be zero-dimensional. We therefore construct a sequence of ideal sheaves $\cI_k \subset \O_{\wb X}$ which on $X$ restricts to the sequence~\eqref{eqn:chain} and which has the property that all $\sL \tensor \cI_k$ are globally generated on the open set $X \subset \wb X$, for some ample line bundle $\sL$ on $\wb X$. Then we conclude as before.

\subsection*{Acknowledgements}

I would like to thank Thomas Peternell and Stefano Urbinati for stimulating discussions on the subject. Stefano Urbinati kindly read a draft version of this paper.
Furthermore I would like to thank Fabrizio Catanese for explaining Example~\ref{exm:more cones} to me.

\section{Notation and conventions} \label{sec:notation}

We work over the field of complex numbers $\C$ throughout.
A variety is an integral separated scheme of finite type over $\C$.
A \emph{pair} $(X, Z)$ consists of a normal variety $X$ and a proper closed subscheme $Z \subsetneq X$.
Unless otherwise specified, by a divisor on a normal variety $X$ we mean a Weil divisor with integer coefficients. For $k \in \{ \Z, \Q, \R \}$, a $k$-divisor is a Weil divisor with coefficients in $k$. The group of $\Z$-divisors on $X$ modulo linear equivalence is denoted $\Cl(X)$.

\subsection*{Boundaries}

A \emph{boundary} on a normal variety $X$ is an effective \Q-divisor $\Delta$ such that $K_X + \Delta$ is \Q-Cartier. In this case we will say that $(X, \Delta)$ is a \emph{log pair}. Note that the coefficients of $\Delta$ may be larger than $1$. If $f\!: Y \to X$ is a proper birational morphism from a normal variety $Y$, we write
\[ K_{Y/X}^\Delta := K_Y + f^{-1}_* \Delta - f^*(K_X + \Delta) \]
for the relative canonical divisor of $(X, \Delta)$.

\subsection*{Reflexive sheaves}

If $\sF$ is a coherent sheaf on a normal variety $X$, we denote its reflexive hull (i.e.~its double dual) by $\sF^{**}$. The torsion subsheaf of $\sF$ is denoted $\tor \sF$. We often write $\factor{\sF}{\tors}$ as a shorthand for $\factor{\sF}{\tor \sF}$.

\subsection*{Log resolutions}

We will need to discuss log resolutions of different kinds of objects, all of which are of course variations on the same theme. For the existence of such resolutions, we refer to~\cite[Thm.~4.2]{dFH09}. Let $X$ be a normal variety. A log resolution of $X$ is a proper birational morphism $f\!: \wt X \to X$, where $\wt X$ is smooth and $\Exc(f)$, the exceptional locus of $f$, is a divisor with simple normal crossings (snc, for short). A log resolution of a log pair $(X, \Delta)$ is a log resolution of $X$ such that $\Exc(f) \cup f^{-1}_* \Delta$ has snc support.

The constant sheaf of rational functions on $X$ is denoted by $\cK_X$. A \emph{fractional ideal sheaf} $\mf a \subset \cK_X$ on $X$ is a coherent $\O_X$-submodule of $\cK_X$. A log resolution of a nonzero fractional ideal sheaf $\mf a \subset \cK_X$ is a log resolution $f\!: \wt X \to X$ of $X$ such that $\mf a \cdot \O_{\wt X} = \O_{\wt X}(E) \subset \cK_{\wt X}$ for some Cartier divisor $E$ on $\wt X$, where $\Exc(f) \cup \supp(E)$ has simple normal crossings. Note that $\mf a \cdot \O_{\wt X} \isom \factor{f^* \mf a}{\tors}$, i.e.~we do not take the reflexive hull.
A log resolution of a subscheme $Z \subsetneq X$ is a log resolution of the ideal sheaf $\mf a \subset \O_X \subset \cK_X$ of $Z$. That is, the scheme-theoretic inverse image $f^{-1}(Z) \subset \wt X$ is required to be a Cartier divisor whose support has simple normal crossings with $\Exc(f)$.

If we need to simultaneously resolve several objects of the types described above, we will talk about \emph{joint log resolutions}. E.g.~using the above notation, a joint log resolution of $(X, \Delta)$, $\mf a$, and $Z$ would be a log resolution of each of the three objects such that $\Exc(f) \cup f^{-1}_* \Delta \cup \supp(E) \cup f^{-1}(Z)$ has snc support.

\begin{rem}
If $Z \subset X$ is a Weil divisor such that $K_X + Z$ is \Q-Cartier, then being a log resolution of the subscheme $Z \subset X$ is a slightly stronger notion than being a log resolution of the log pair $(X, Z)$. We trust that this will not lead to any confusion.
\end{rem}

\subsection*{Global generation of sheaves}

If $\sF$ is a coherent sheaf on a variety $X$ and $x \in X$ is a point, we say that $\sF$ is globally generated at $x$ if the natural morphism
\[ H^0(X, \sF) \tensor_\C \O_X \lto \sF \]
is surjective at $x$. We say that $\sF$ is globally generated on an open subset $U \subset X$ if $\sF$ is globally generated at every point $x \in U$.


\subsection*{Relative N\'eron--Severi spaces}

Let $f\!: Y \to X$ be a proper morphism between normal varieties, where $Y$ is \Q-factorial. We denote by $N^1(Y/X)_\Q$ the vector space of \Q-divisors on $Y$ modulo $f$-numerical equivalence (that is, numerical equivalence on all curves contracted by~$f$). We denote by $N_1(Y/X)_\Q$ the vector space spanned by all curves contracted by $f$, modulo numerical equivalence. These two vector spaces are finite-dimensional, and dual to each other via the intersection pairing
\[ N^1(Y/X)_\Q \x N_1(Y/X)_\Q \lto \Q. \]
Their common dimension is called the relative Picard number $\rho(Y/X)$ of $f$.

\section{Pullbacks of Weil divisors and numerical \Q-factoriality} \label{sec:Weil pullback}

In order to define multiplier ideals in the absence of any \Q-Gorenstein assumption, it is necessary to dispose of a notion of pullback for arbitrary Weil divisors. Such a pullback was introduced in~\cite{dFH09} and rephrased using the language of nef envelopes in~\cite{BdFF12}.
For a certain class of Weil divisors, called \emph{numerically \Q-Cartier} divisors, this pullback is particularly well-behaved. Numerically \Q-Cartier divisors were introduced in~\cite{BdFF12} from a $b$-divisorial point of view and reinterpreted in a somewhat more down-to-earth fashion in~\cite{BdFFU13}.
For ease of reference, we recall in this section the relevant facts and definitions. For more details and full proofs, we refer the reader to the original papers cited above.

\begin{rem}
Closely related notions of pullback and numerical \Q-Cartierness were already discussed by Nakayama~\cite[Ch.~II, Lem.~2.12 and Ch.~III, Cor.~5.11]{Nak04}.
\end{rem}

\subsection{Pulling back Weil divisors}

Consider a proper birational morphism $f\!: Y \to X$ between normal varieties, and let $D$ be a Weil divisor on $X$. The \emph{natural pullback} $f^\ntl D$ of $D$ along $f$ is defined by
\[ \O_Y(-f^\ntl D) = ( \O_X(-D) \cdot \O_Y )^{**}, \]
where we consider $\O_X(-D) \subset \cK_X$ as a fractional ideal sheaf on $X$.
We define the \emph{pullback} $f^* D$ of $D$ along $f$ by setting
\[ f^* D := \lim_{k \conv \infty} \frac{f^\ntl(kD)}{k} = \inf_{k \ge 1} \frac{f^\ntl(kD)}{k}, \]
where the limit and the infimum are to be understood coefficient-wise. The limits actually exist in $\R$, hence $f^* D$ is a well-defined $\R$-divisor.
If $D$ is \Q-Cartier, then this definition agrees with the usual notion of pullback. We have $f^*(mD) = m \cdot f^* D$ for any integer $m \ge 0$. Thus we can extend the map $f^*$ to arbitrary \Q-divisors by clearing denominators. Furthermore, we have $f^*(-D) \ge -f^* D$. However, this inequality may be strict -- see Section~\ref{sec:num Q-Cartier} below.

\subsection{Relative canonical divisors}

Let $f\!: Y \to X$ be as before, and fix canonical divisors $K_X$ and $K_Y$ on $X$ and on $Y$, respectively, that satisfy $K_X = f_* K_Y$. For any $m \ge 1$, we define the \emph{$m$-th limiting relative canonical divisor} of $Y$ over $X$ by
\[ K_{Y/X, m} := K_Y - \frac{1}{m} f^\ntl(mK_X) \]
and the \emph{relative canonical divisor} of $Y$ over $X$ by
\[ K_{Y/X}^- := K_Y - f^* K_X = \lim_{m \conv \infty} K_{Y/X, m}. \]
For any $m$, we have the inequality $K_{Y/X, m} \le K_{Y/X}^-$. If $K_X$ is \Q-Cartier, then $K_{Y/X}^-$ coincides with the usual relative canonical divisor $K_{Y/X}$. Note, however, that here we have not defined the symbol $K_{Y/X}$ in general.

\subsection{Numerically \Q-Cartier divisors} \label{sec:num Q-Cartier}

We next address the question of when the equality $f^*(-D) = -f^* D$ holds.

\begin{prp}[Numerically \Q-Cartier divisors] \label{prp:num Q-Cartier}
Let $X$ be a normal variety and $D$ a \Q-divisor on $X$. Then the following conditions are equivalent:
\begin{enumerate}
\item\label{itm:num Q.1} There is a proper birational morphism $f\!: Y \to X$, where $Y$ is \Q-factorial, and an $f$-numerically trivial \Q-divisor $D'$ on $Y$ such that $f_* D' = D$.
\item\label{itm:num Q.2} For any proper birational morphism $f\!: Y \to X$ where $Y$ is \Q-factorial, there is an $f$-numerically trivial \Q-divisor $D'$ on $Y$ such that $f_* D' = D$.
\item\label{itm:num Q.4} For any proper birational morphism $f\!: Y \to X$, we have $f^*(-D) = -f^* D$.
\end{enumerate}
\end{prp}

\begin{proof}
See~\cite[Props.~5.3, 5.9]{BdFFU13}.
\end{proof}

\begin{dfn}[Numerical \Q-factoriality] \label{dfn:num Q-fact}
A \Q-divisor $D$ on a normal variety $X$ is called \emph{numerically \Q-Cartier} if the equivalent conditions of Proposition~\ref{prp:num Q-Cartier} are satisfied.
The vector space of \Q-divisors on $X$ modulo the subspace of numerically \Q-Cartier divisors is denoted by $\Cln(X)_\Q$.
The variety $X$ is called \emph{numerically \Q-factorial} if $\Cln(X)_\Q = 0$, i.e.~if every \Q-divisor on $X$ is numerically \Q-Cartier.
\end{dfn}

\begin{rem} \label{rem:Knum ratl}
In~(\ref{prp:num Q-Cartier}.\ref{itm:num Q.1}) and~(\ref{prp:num Q-Cartier}.\ref{itm:num Q.2}), the divisor $D'$ is unique, and $D' = f^* D$. Hence for a numerically \Q-Cartier divisor $D$, the $\R$-divisor $f^* D$ is in fact a \Q-divisor. In particular, if $K_X$ is numerically \Q-Cartier, then for any normal modification $Y \to X$ the relative canonical divisor $K_{Y/X}^-$ is a \Q-divisor.
\end{rem}

\begin{exm}[Surfaces] \label{exm:surfaces}
Any normal surface is numerically \Q-factorial. This follows from the existence of Mumford's pullback for divisors on surfaces~\cite[Ch.~4.1]{KM98}.
\end{exm}

\begin{exm}[Cones] \label{exm:cones}
Let $Y$ be a smooth projective variety and $L$ an ample divisor on $Y$. We consider the \emph{affine cone} $X = C_a(Y, L)$ over $Y$ with respect to $L$, defined as the spectrum of the section ring
\[ R(X, \sL) := \bigoplus_{n \ge 0} H^0(Y, nL). \]
This is a generalization of the classical affine cone over a variety embedded in projective space.

For any prime divisor $D$ on $Y$, the cone $C(D) = C_a(D, L|_D)$ is a divisor on $X$. The map $D \mapsto C(D)$ extends linearly to give an isomorphism
\[ \Cl(X) \isom \factor{\Pic(Y)}{\Z \cdot L.} \]
The divisor $C(D)$ is \Q-Cartier if and only if $D$ and $L$ are \Q-linearly proportional, while $C(D)$ is numerically \Q-Cartier if and only if $D$ and $L$ are numerically proportional~\cite[Ex.~2.31, Lem.~2.32]{BdFF12}.
The canonical divisor of $X$ is given by $K_X = C(K_Y)$.

It follows that $X$ is \Q-factorial if and only if on $Y$, numerical equivalence coincides with \Q-linear equivalence, and the Picard number $\rho(Y) = 1$. The first condition is well-known to be equivalent to $H^1(Y, \O_Y) = 0$. On the other hand, $X$ is numerically \Q-factorial if and only if $\rho(Y) = 1$.
\end{exm}

\begin{exm}[More cones] \label{exm:more cones}
This example is a continuation of the previous one. We make the following claim. \\[1ex]
\emph{In every dimension $n \ge 2$, there exist isolated singularities that are numerically \Q-factorial, but not \Q-factorial (more precisely, the canonical divisor is not \Q-Cartier).} \\[1ex]
Namely, if $Y$ is smooth projective of dimension $n - 1$ and Picard number one, $K_Y$ is ample, and $H^1(Y, \O_Y) \ne 0$, let $B \in \Pic^0(Y)$ be a numerically trivial non-torsion divisor on $Y$. Then the cone $X = C_a(Y, K_Y + B)$ over $Y$ has the required properties.

Concerning the existence of $Y$, if $n = 2$ we may simply take $Y$ to be a curve of genus $g \ge 2$. If $n \ge 3$, let $A$ be an abelian $n$-fold of Picard number one~\cite[Ex.~1.2.26]{Laz04a}, and take $Y \subset A$ to be a smooth ample divisor. Then $K_Y$ is ample by the adjunction formula and $H^1(Y, \O_Y) = H^1(A, \O_A) \ne 0$ by the weak Lefschetz theorem~\cite[Ex.~3.1.24]{Laz04a}. If $n \ge 4$, then $\Pic A \to \Pic Y$ is an isomorphism~\cite[Ex.~3.1.25]{Laz04a}, hence $\rho(Y) = \rho(A) = 1$. In case $n = 3$, at least if $Y$ is very general in a sufficiently ample linear system the infinitesimal Noether--Lefschetz theorem implies that
\[ H^{1,1}(A) \cap H^2(A, \Q) \lto H^{1,1}(Y) \cap H^2(Y, \Q) \]
is surjective~\cite[Cor.~on p.~179]{CGGH83}. By the Lefschetz $(1, 1)$-theorem~\cite[Rem.~1.1.21]{Laz04a}, we again obtain $\rho(Y) \le \rho(A) = 1$.
\end{exm}

We end this section with some facts about numerically \Q-Cartier divisors which we will use later.

\begin{prp}[Short exact sequence] \label{prp:ses}
Let $f\!: Y \to X$ be a proper birational morphism, where $Y$ is \Q-factorial, and let $E_1, \dots, E_\ell$ be the divisorial components of the exceptional locus of $f$. Then we have a short exact sequence of \Q-vector spaces
\[ 0 \lto \bigoplus_{i=1}^\ell \Q \cdot [E_i] \lto N^1(Y / X)_\Q \lto \Cln(X)_\Q \lto 0 \]
induced by the pushforward of divisors along $f$.
\end{prp}

\begin{proof}
See~\cite[Cor.~5.4.ii)]{BdFFU13}.
\end{proof}

\begin{thm}[Rational singularities] \label{thm:ratl}
If $X$ has rational singularities, then the notions of \Q-Cartier and numerically \Q-Cartier divisors coincide, i.e.~any numerically \Q-Cartier divisor is already \Q-Cartier.
\end{thm}

\begin{proof}
See~\cite[Thm.~5.11]{BdFFU13}.
\end{proof}

\begin{rem}
In Theorem~\ref{thm:ratl} it is sufficient to assume that $X$ has \emph{1-rational singularities}, meaning that $R^1 f_* \O_{\wt X} = 0$ for some (equivalently, any) resolution $f\!: \wt X \to X$. This can be seen from the proof of~\cite[Thm.~5.11]{BdFFU13}.
\end{rem}

\section{Non-\Q-Gorenstein multiplier ideals} \label{sec:non-Q-Gor mult ideals}

Following~\cite{dFH09}, we recall in this section how to associate a multiplier ideal $\cJ(X, cZ)$ to a pair $(X, Z)$ and a real number $c \ge 0$.
If $\mf a \subset \O_X$ is the ideal sheaf of the subscheme $Z$, the multiplier ideal may also be denoted by $\cJ(X, \mf a^c)$.

\subsection{The classical case}

Classically, one defines multiplier ideals under some \Q-Gorenstein condition. More precisely, assume that a boundary $\Delta$ on $X$ is given, i.e.~the divisor $K_X + \Delta$ is \Q-Cartier. Let $f\!: Y \to X$ be a joint log resolution of $(X, \Delta)$ and $\mf a$, and write $\mf a \cdot \O_Y = \O_Y(-D)$ with $D$ an effective Cartier divisor. Then one defines
\[ \cJ((X, \Delta), cZ) := f_* \O_Y( \rp K_Y - f^*(K_X + \Delta) - cD. ). \]
One checks that this definition is independent of the choice of log resolution. See~\cite[Def.~9.3.60]{Laz04b} for more details.

\subsection{The general case}

Let $(X, Z)$ be a pair and $c \ge 0$ a real number. Fix a canonical divisor $K_X$ on $X$. For any natural number $m \ge 1$, consider a joint log resolution $f\!: Y \to X$ of $Z \subset X$ and $\O_X(-mK_X) \subset \cK_X$. Note that $f$ may heavily depend on $m$. We define the \emph{$m$-th limiting multiplier ideal sheaf} of $(X, cZ)$ as
\[ \cJ_m(X, cZ) := f_* \O_Y( \rp K_{Y/X, m} - cD. ) \subset \O_X, \]
where $D = f^{-1}(Z)$. This coherent ideal sheaf is independent of the choice of $K_X$ and $f$.
For any $k, m > 0$, one has an inclusion $\cJ_m(X, cZ) \subset \cJ_{km}(X, cZ)$. By the Noetherian property of $\O_X$, it follows that the set $\{ \cJ_m(X, cZ) \}_{m \ge 1}$ has a unique maximal element.
We define the \emph{multiplier ideal sheaf} $\cJ(X, cZ)$ to be this maximal element. We have $\cJ(X, cZ) = \cJ_m(X, cZ)$ for $m$ sufficiently divisible. If $K_X$ is \Q-Cartier, then this definition of multiplier ideal agrees with the classical one. More precisely, we have $\cJ((X, 0), cZ) = \cJ_m(X, cZ)$ as soon as $mK_X$ is Cartier.

By~\cite[Prop.~4.9]{dFH09}, we have that $\cJ(X, cZ) \subset \cJ(X, c'Z)$ for $c' < c$ and that $\cJ(X, cZ) = \cJ(X, (c + \eps) Z)$ for $\eps > 0$ sufficiently small, depending on $c$. Hence the following definition makes sense.

\begin{dfn}[Jumping numbers] \label{dfn:jumping numbers}
A positive real number $\xi > 0$ is called a \emph{jumping number} of the pair $(X, Z)$ if
\[ \cJ(X, \xi Z) \subsetneq \cJ(X, \lambda Z) \quad \text{for all $0 \le \lambda < \xi$}. \]
We denote the set of jumping numbers of $(X, Z)$ by $\Xi(X, Z) \subset \R^+$.

\end{dfn}

\subsection{Compatible boundaries}

The following definition~\cite[Def.~5.1]{dFH09} serves to relate the general case just described to the classical one.

\begin{dfn}[Compatible boundary] \label{dfn:comp bd}
Let $(X, Z)$ be a pair, and fix an integer $m \ge 2$. A boundary $\Delta$ on $X$ is said to be \emph{$m$-compatible} for $(X, Z)$ if there exists a canonical divisor $K_X$ on $X$ and a joint log resolution $f\!: Y \to X$ of $(X, \Delta)$, $Z \subset X$ and $\O_X(-mK_X)$ such that:
\begin{enumerate}
\item the divisor $m\Delta$ is integral, and $\rd \Delta. = 0$,
\item no component of $\Delta$ is contained in the support of $Z$, and
\item the equality $K_{Y/X}^\Delta = K_{Y/X, m}$ holds.
\end{enumerate}
\end{dfn}

The point of this definition is the following observation.

\begin{prp}[Realizing multiplier ideals as classical ones] \label{prp:realizing}
Let $(X, Z)$ be a pair and $c \ge 0$ a real number. Choose an integer $m$ such that $\cJ(X, cZ) = \cJ_m(X, cZ)$, and let $\Delta$ be an $m$-compatible boundary for $(X, Z)$. Then we have
\[ \cJ(X, cZ) = \cJ((X, \Delta), cZ). \]
\end{prp}

\begin{proof}
See~\cite[Prop.~5.2]{dFH09}.
\end{proof}

Of course, the usefulness of this notion depends on the existence of compatible boundaries. This was established in~\cite[Thm.~5.4]{dFH09}. The proof given there is constructive and yields the following more precise result, which we record for later use.

\begin{thm}[Existence of compatible boundaries] \label{thm:ex comp bd}
Let $(X, Z)$ be a pair, and let $m \ge 2$ be a natural number. Choose an effective Weil divisor $D$ on $X$ such that $K_X - D$ is Cartier, and let $\sL \in \Pic X$ be a line bundle such that $\sL(-mD) := \sL \tensor \O_X(-mD)$ is globally generated. Pick a finite-dimensional subspace $V \subset H^0(X, \sL(-mD))$ that generates $\sL(-mD)$, and let $M$ be the divisor of a general element of $V$. Then
\[ \Delta := \frac{1}{m} M \]
is an $m$-compatible boundary for $(X, Z)$. \qed
\end{thm}

It follows that for any pair $(X, Z)$ and $c \ge 0$, the set of ideal sheaves
\[ \big\{ \cJ((X, \Delta), cZ) \;\big|\; \text{$\Delta$ a boundary on $X$ in the sense of Section \ref{sec:notation}} \big\} \]
has a unique maximal element, namely $\cJ(X, cZ)$.
This reduces some, but by no means all, questions about multiplier ideals to the classical case. Roughly speaking, as long as one is dealing with only finitely many values of $(X, Z)$ and $c$, one may choose the number $m$ sufficiently divisible such that $\cJ = \cJ_m$ for all those pairs, and then one picks an $m$-compatible boundary. In general, however, it is hard to tell where the sequence of limiting multiplier ideals stabilizes. Even if the pair $(X, Z)$ is fixed, given a collection of infinitely many values of $c$, the required $m$'s might become arbitrarily large. One then needs to consider infinitely many different compatible boundaries. If instead one tries to work directly with the definitions, one needs to consider infinitely many resolutions of $X$.

However, if $K_X$ is numerically \Q-Cartier, we can circumvent these difficulties thanks to the following theorem.

\begin{thm}[Multiplier ideals on numerically \Q-Gorenstein varieties] \label{thm:mult ideals num Q-Gor}
Let $(X, Z)$ be a pair, where $K_X$ is numerically \Q-Cartier. Choose a log resolution $f\!: \wt X \to X$ of $(X, Z)$. Then for any $t > 0$ we have
\[ \cJ(X, tZ) = f_* \O_{\wt X} \big( \rp K_{\wt X/X}^- - t \cdot f^{-1}(Z). \big). \]
\end{thm}

\begin{proof}
See~\cite[Thm.~1.3]{BdFFU13}.
\end{proof}

\section{Elementary properties of jumping numbers}

The aim of the present section is to prove Proposition~\ref{prp:elem}. So let $(X, Z)$ be a pair. We denote the ideal sheaf of $Z$ by $\mf a \subset \O_X$.

\subsection{Nonemptiness}

Pick a canonical divisor $K_X$ on $X$ and a number $m$ such that $\cJ(X, \emptyset) = \cJ_m(X, \emptyset)$. If $f\!: Y \to X$ is a joint log resolution of $(X, Z)$ and $\O_X(-mK_X)$, then
\[ \cJ(X, \emptyset) = f_* \O_Y( \rp K_{Y/X, m}. ). \]
As $Z \ne \emptyset$, we have $\mf a \cdot \O_Y = \O_Y(-D)$, where $D$ is a nonzero effective Cartier divisor on $Y$.
Pick $x \in \supp(Z)$ arbitrarily, and let $E \subset \supp(D)$ be a prime divisor with $x \in f(E)$. Take any nonzero function germ $h \in \cJ(X, \emptyset)_x$. Then for $\xi \gg 0$, the coefficient of $E$ in $-\rp K_{Y/X}^- - \xi D.$ will be larger than the order of vanishing of $f^* h$ along $E$.
Fix such a number $\xi$. Let $\ell$ be such that $\cJ(X, \xi Z) = \cJ_\ell(X, \xi Z)$. Let $g\!: \wt Y \to X$ be a joint log resolution of $(X, Z)$ and $\O_X(-\ell K_X)$ that factors through $f$. Then
\[ \cJ(X, \xi Z) = g_* \O_{\wt Y}( \rp K_{\wt Y/X, \ell} - \xi D. ). \]
As $K_{\wt Y/X, \ell} \le K_{\wt Y/X}^-$, we see that $g^* h$ is not contained in $\O_{\wt Y}( \rp K_{\wt Y/X, \ell} - \xi D. )$ at the generic point of the strict transform of $E$ on $\wt Y$. Hence $h$ is not contained in $\cJ(X, \xi Z)_x$.
It follows that the inclusion $\cJ(X, \xi Z) \subsetneq \cJ(X, \emptyset)$ is strict, and so there must be a jumping number of the pair $(X, Z)$ in the interval $(0, \xi]$.

\subsection{Unboundedness}

Let $\xi \in \Xi(X, Z)$ be a jumping number. Spelled out explicitly, this means that for any $0 < \eps \le \xi$ there is a point $x \in X$ and a function $h \in \O_{X,x}$, both depending on $\eps$, such that $h \in \cJ(X, (\xi - \eps)Z)_x$, but $h \not\in \cJ(X, \xi Z)_x$.

Fix an arbitrary $0 < \eps \le \xi$. We regard a germ $g \in \mf a_x$ as a regular function on a small neighborhood $U$ of $x$. For any proper birational morphism $f\!: Y \to X$ such that $Y$ is normal and $\mf a \cdot \O_Y = \O_Y(-D)$ is invertible, we may then write
\[ \dv(f^* g) = D|_V + M_g, \]
where $V = f^{-1}(U)$ and $M_g$ is an effective divisor on $V$.
Here we take the divisor of $f^* g$ as a function, not as a section of $\O_Y(-D)$.

Choose a finite set of generators for the ideal $\mf a_x = \langle g_1, \dots, g_r \rangle$ in the noetherian ring $\O_{X,x}$. Then after possibly shrinking $U$, the sheaf $\O_Y(-D)|_V$ will be generated by the sections $f^* g_1, \dots, f^* g_r$.
It follows that if $g = \sum_{i=1}^r \lambda_i g_i \in \mf a_x$ is a general $\C$-linear combination of the chosen generators, then $M_g$ and $D|_V$ do not have any common components.

Now fix such a general $g \in \mf a_x$. As $h \in \cJ_m(X, (\xi - \eps)Z)_x$ for some $m$, we get
\[ g \cdot h \in \cJ_m(X, (\xi + 1 - \eps)Z)_x \subset \cJ(X, (\xi + 1 - \eps)Z)_x \]
from the definitions. On the other hand, we will show that $g \cdot h \not\in \cJ(X, (\xi + 1)Z)_x$. Proceeding by contradiction, assume that $g \cdot h \in \cJ_m(X, (\xi + 1)Z)_x$ for some $m$. Since $M_g$ and $D|_V$ do not share any common components, it follows that
\[ h \in \cJ_m(X, \xi Z)_x \subset \cJ(X, \xi Z)_x, \]
contradicting the definition of $h$. So $\cJ(X, (\xi + 1)Z) \subsetneq \cJ(X, (\xi + 1 - \eps)Z)$. As $\eps$ was chosen arbitarily, this proves that $\xi + 1$ is a jumping number.

The second statement of~(\ref{prp:elem}.\ref{itm:elem.unb}) follows from what we have just proved, combined with~(\ref{prp:elem}.\ref{itm:elem.emp}).

\subsection{DCC property}

Any decreasing sequence $\xi_1 > \xi_2 > \cdots$ contained in $\Xi(X, Z)$ would give rise to a strictly ascending chain of coherent ideal sheaves
\[ \cJ(X, \xi_1 Z) \subsetneq \cJ(X, \xi_2 Z) \subsetneq \cdots, \]
which is impossible by the Noetherian property of $\O_X$. For the second statement, assuming it were false we can easily construct a decreasing sequence as above.

\subsection{Countability}

Consider the following subset of $\R$:
\[ A := \big\{ a \in \R \;\big|\; \text{$\Xi(X, Z) \cap [0, a]$ is countable} \big\}. \]
As $0 \in A$, this set is nonempty. If $(a_n) \subset A$ is a sequence converging to some number $a$, it is immediate that also $a \in A$. Hence $A$ is closed. By the second part of~(\ref{prp:elem}.\ref{itm:elem.dcc}), the set $A$ is open. It follows that $A = \R$, and then $\Xi(X, Z) = \bigcup_{n \in \N} \big( \Xi(X, Z) \cap [0, n] \big)$ is countable.

\begin{rem}
Of course, the proof just given has nothing to do with jumping numbers, and it shows quite generally that any DCC set is countable.
\end{rem}

\subsection{Periodicity}

By~(\ref{prp:elem}.\ref{itm:elem.unb}), it suffices to show that if $\xi > \dim X - 1$ is not a jumping number, then also $\xi + 1$ is not a jumping number.
So we assume that $\cJ(X, (\xi - \eps)Z) = \cJ(X, \xi Z)$ for sufficiently small $\eps > 0$, and we need to show $\cJ(X, (\xi + 1 - \eps)Z) = \cJ(X, (\xi + 1)Z)$ for small $\eps > 0$.

\begin{clm}[Skoda's theorem] \label{clm:Skoda}
For any $c \ge \dim X - 1$, we have
\[ \cJ(X, (c + 1)Z) = \mf a \cdot \cJ(X, cZ). \]
\end{clm}

\begin{proof}[Proof of Claim~\ref{clm:Skoda}]
Fix a number $m$ such that $\cJ(X, cZ) = \cJ_m(X, cZ)$ and $\cJ(X, (c + 1)Z) = \cJ_m(X, (c + 1)Z)$, and choose an $m$-compatible boundary $\Delta$ for the pair $(X, Z)$. We then have
\[ \cJ(X, cZ) = \cJ((X, \Delta), cZ) \quad \text{and} \quad
\cJ(X, (c + 1)Z) = \cJ((X, \Delta), (c + 1)Z) \]
by Proposition~\ref{prp:realizing}. The claim now follows from Skoda's theorem~\cite[Thm.~9.6.21.ii)]{Laz04b}, asserting that
\[ \cJ((X, \Delta), (c + 1)Z) = \mf a \cdot \cJ((X, \Delta), cZ). \]
Note that~\cite[Thm.~9.6.21]{Laz04b} is only applicable if $X$ is smooth, $\Delta = 0$, and $c$ is rational. However, as explained in Remark~9.6.23 and after Variant~9.6.39 of~\cite{Laz04b}, the statement remains true in the generality required here. Indeed, upon replacing the relative vanishing of~\cite[Variant~9.4.4]{Laz04b} by Thm.~9.4.17 of that book, the proof of Skoda's theorem goes through verbatim, even if $c$ is not rational.
\end{proof}

Returning to the proof of~(\ref{prp:elem}.\ref{itm:elem.per}), for $\eps > 0$ sufficiently small Claim~\ref{clm:Skoda} yields
\[ \cJ(X, (\xi + 1 - \eps)Z) = \mf a \cdot \cJ(X, (\xi - \eps)Z) = \mf a \cdot \cJ(X, \xi Z) = \cJ(X, (\xi + 1)Z), \]
finishing the proof of~(\ref{prp:elem}.\ref{itm:elem.per}) and the whole Proposition~\ref{prp:elem}. \qed

\section{Resolutions of sheaves} \label{sec:sheaf resol}

As a preparation for the proof of Theorem~\ref{thm:discrete isol}, we recall here the notion of resolution morphism for a coherent sheaf (not to be confused with resolutions in the sense of homological algebra) and a positivity property of such resolutions which will prove crucial for us.

The construction presented here is essentially due to Rossi~\cite[Thm.~3.5]{Ros68} in the more general context of coherent analytic sheaves on complex spaces. However, as the positivity property mentioned above is not addressed in~\cite{Ros68}, we have chosen to include here a full account of the construction of sheaf resolutions (in the algebraic case). Our proof is somewhat simpler than the original argument. Furthermore, it answers a question of Rossi~\cite[p.~72]{Ros68} asking for a universal property of his construction.

\begin{dfn}[Resolution of a sheaf] \label{dfn:resol sheaf}
Let $\sF$ be a coherent sheaf on a normal variety $X$. A \emph{resolution} of $\sF$ is a proper birational morphism $f\!: Y \to X$ such that $Y$ is normal and $\factor{f^* \sF}{\tors}$ is locally free.
A resolution $f$ of $\sF$ is called \emph{minimal} if every other resolution of $\sF$ factors through $f$.
\end{dfn}

Note that a resolution of a sheaf $\sF$ on $X$ will usually not be a resolution of singularities for $X$. Of course, the minimal resolution of $\sF$ is unique if it exists.

\begin{thm}[Existence of resolutions] \label{thm:exist resol}
Let $X$ be a normal variety and $\sF$ a coherent sheaf on $X$. Then the minimal resolution $f\!: Y \to X$ of $\sF$ exists. Furthermore, if $\sF$ has rank one, then $\factor{f^* \sF}{\tors}$ is an $f$-ample invertible sheaf.
\end{thm}

\begin{proof}
First we construct the required resolution locally. After shrinking $X$, we may assume that there is a surjection $\alpha\!: \O_X^{\oplus p} \surj \sF$ for some integer $p$. Let $X^\circ$ be the open subset of $X$ where $\sF$ is locally free, say of rank $r$. Then $\alpha$ determines a morphism $F\!: X^\circ \to G := \Gr(p, r)$ into the Grassmannian of $r$-dimensional quotients of $\C^p$. We may view $F$ as a rational map $X \dashrightarrow G$.

Let $\pi\!: \wt X \to X$ be a resolution of indeterminacy for $F$, i.e.~a blowup such that $F$ extends to a morphism $\phi\!: \wt X \to G$.
\[ \xymatrix{
\wt X \ar_\pi[d] \ar^\phi[drr] & & \\
X \ar@{-->}^F[rr] & & G
} \]
Let $\sO_G^{\oplus p} \surj \sE$ be the tautological quotient bundle on $G$. Then we have a surjection $\pi^* \sF \surj \phi^* \sE$. It is an isomorphism on $\pi^{-1}(X^\circ)$, hence it induces an isomorphism $\factor{\pi^* \sF}{\tors} \isom \phi^* \sE$. In particular, $\factor{\pi^* \sF}{\tors}$ is locally free.
Conversely, if $\pi\!: \wt X \to X$ is a blowup such that $\factor{\pi^* \sF}{\tors}$ is locally free, then the surjection
\[ \sO_{\wt X}^{\oplus p} \xrightarrow{\pi^* \alpha} \pi^* \sF \lto \factor{\pi^* \sF}{\tors} \]
provides a morphism $\wt X \to G$ that extends $F$.

Now let $Y' \subset X \x G$ be the closure of the graph of $F$, and let $Y$ be the normalization of $Y'$.
\[ \xymatrix{
Y \ar^\nu[r] \ar@/_/_f[dr] & Y' \ar_{\pr_1}[d] \ar^{\pr_2}[drr] & & \\
& X \ar@{-->}^F[rr] & & G
} \]
Then $f = \pr_1 \circ\, \nu\!: Y \to X$ is a resolution of indeterminacy for $F$, and every other such resolution $\pi\!: \wt X \to X$ with $\wt X$ normal factors through $f$.
By what we have observed above, this means that $f$ is a minimal resolution of $\sF$. By uniqueness of the minimal resolution, $f$ does not depend on the surjection $\alpha$ chosen in the beginning. In particular, the local constructions glue to give a globally defined minimal resolution of $\sF$. This proves the first half of the theorem.

For the second statement, assume that $\sF$ has rank one. Since the statement is local, we may again assume that we are in the local situation described above, and we continue to use that notation. Note that since $\sF$ has rank one, $G = \Gr(p, 1) = \P^{p-1}$. So
\[ \factor{\pr_1^* \sF}{\tors} \isom \pr_2^* \O_{\P^{p-1}}(1). \]
It follows immediately that $\factor{\pr_1^* \sF}{\tors}$ is ample (even very ample) on the fibres of $\pr_1$. As ampleness is preserved under finite pullbacks, pulling back everything to the normalization $Y$ of $Y'$ we get that $\factor{f^* \sF}{\tors}$ is ample on the fibres of $f$. By~\cite[Thm.~1.7.8]{Laz04a}, this implies the $f$-ampleness of $\factor{f^* \sF}{\tors}$.
\end{proof}

\begin{rem} \label{rem:alternative}
If $\sF$ has rank one (which is the only case we shall need), an alternative approach is as follows. Replacing $\sF$ by $\factor{\sF}{\tors}$, we may assume that $\sF$ is torsion-free. Then $\sF$ is isomorphic to a fractional ideal sheaf $\mf a \subset \cK_X$. Let $f\!: Y \to X$ be the normalization of the blowing-up of $\mf a$, which can be defined in exactly the same fashion as the blowing-up of an ordinary ideal sheaf~\cite[Ch.~II, Sec.~7]{Har77}. Then the assertions we need to prove follow directly from the analogues of~\cite[Ch.~II, Props.~7.14, 7.10]{Har77}.
\end{rem}

\begin{rem}
If $\sF$ has rank $r \ge 2$, then there is no resolution $f$ of $\sF$ such that $\factor{f^* \sF}{\tors}$ is an $f$-ample vector bundle. For the minimal resolution, this can be seen from the proof of Theorem~\ref{thm:exist resol} and the fact that the tautological quotient bundle on the Grassmannian $\Gr(p, r)$ is not ample if $r \ge 2$~\cite[Ex.~6.1.6]{Laz04b}. As any resolution factors through the minimal one, the assertion follows.
\end{rem}

\begin{rem}
It is clear that both the proof of Theorem~\ref{thm:exist resol} as well as the alternative approach outlined in Remark~\ref{rem:alternative} work over an algebraically closed field of arbitrary characteristic.
\end{rem}

\section{Global generation of Weil divisorial sheaves}

The purpose of the present section is to prove the following theorem, which is a special case of Conjecture~\ref{conj:glob gen} and generalizes the previously known special case~\cite[Prop.~5.5]{Urb12}, where $X$ is required to have isolated singularities.

\begin{thm}[Global generation for isolated non-\Q-Cartier loci] \label{thm:glob gen isol}
Let $X$ be a normal projective variety and $D$ a Weil divisor on $X$, with non-\Q-Cartier locus $W \subset X$.
Define the open set $U \subset X$ as the complement of $W$ union the isolated points of $W$.
Then there is an ample Cartier divisor $H$ on $X$ such that for $m \in \N$ sufficiently divisible, the sheaf $\O_X(m(D + H))$ is globally generated on $U$.
\end{thm}

The proof proceeds along the general lines of~\cite{Urb12}, except that instead of passing to a resolution of $X$ and using Kawamata--Viehweg vanishing, we only resolve the sheaf $\O_X(D)$ (in the sense of Section~\ref{sec:sheaf resol}) and use Fujita vanishing.
In particular, all the key ingredients of our proof (the others being relative Serre vanishing and Castelnuovo--Mumford regularity) remain valid in arbitrary characteristic. Hence we see that Theorem~\ref{thm:glob gen isol} is also true in positive characteristic, as one might expect.

\begin{rem}
In Conjecture~\ref{conj:glob gen}, the conclusion holds for all $m \in \N$ and not just for sufficiently divisible $m$. This should also be true of Theorem~\ref{thm:glob gen isol}. We have not paid attention to this, as it is not needed for our purposes and would only complicate the proof.
\end{rem}

Before starting the proof, we record one auxiliary lemma.

\begin{lem}[Extensions of globally generated sheaves] \label{lem:ext glob gen}
Let $X$ be a projective variety, and let
\[ 0 \lto \sF' \lto \sF \lto \sF'' \lto 0 \]
be a short exact sequence of coherent sheaves on $X$. Assume that $\sF'$ is globally generated and that $H^1(X, \sF') = 0$. If $x \in X$ is any point such that $\sF''$ is globally generated at $x$, then so is $\sF$.
\end{lem}

\begin{proof}
We have the following diagram with exact rows:
\[ \xymatrix{
0 \ar[r] & H^0(\sF') \tensor \O_{X, x} \ar[r] \ar@{->>}[d] & H^0(\sF) \tensor \O_{X, x} \ar[r] \ar[d] & H^0(\sF'') \tensor \O_{X, x} \ar[r] \ar@{->>}[d] & 0 \\
0 \ar[r] & \sF'_x \ar[r] & \sF_x \ar[r] & \sF''_x \ar[r] & 0.
} \]
The outer vertical maps are surjective. By the four lemma, so is the middle one.
\end{proof}

\subsection{Proof of Theorem~\ref{thm:glob gen isol}}

For convenience of the reader, the proof is divided into four steps.

\subsubsection*{Step 1: Blowing up}

Let $\wt D = N_0 \cdot D$ be the smallest positive multiple of $D$ which is Cartier outside of $W$, and let $\wt H = N_0 \cdot H$ for some very ample Cartier divisor $H$ on $X$.
Let $f\!: Y \to X$ be the minimal resolution of $\O_X(\wt D)$. Then $f$ is an isomorphism outside of $W$, and
\[ \O_Y(B) := \factor{f^* \O_X(\wt D)}{\tors} \]
is an $f$-ample invertible sheaf by Theorem~\ref{thm:exist resol}. Note that we are free to replace $D$ by $D + N_1 \cdot \wt H$ for any $N_1 > 0$. This does not change $N_0$ and $f$, and it changes $B$ to $B + f^*(N_0 N_1 \cdot \wt H)$. Hence by~\cite[Prop.~1.45]{KM98}, we may assume that the Cartier divisor $B$ is globally ample on $Y$.

\subsubsection*{Step 2: Vanishing theorems}

By Fujita vanishing~\cite[Thm.~1.4.35]{Laz04a}, we have
\[ H^i(Y, \O_Y(mB + P)) = 0 \qquad \text{for $i > 0$, $m \gg 0$ and any nef divisor $P$ on $Y$.} \]
Furthermore, by relative Serre vanishing~\cite[Thm.~1.7.6]{Laz04a}
\[ R^i f_* \O_Y(mB) = 0 \qquad \text{for $i > 0$ and $m \gg 0$.} \]
It now follows from the projection formula for higher direct images~\cite[Ch.~III, Ex.~8.3]{Har77} and the Leray spectral sequence associated to the map $f$ and the sheaf $\O_Y(mB + \ell \cdot f^* \wt H)$ that
\[ H^i(X, f_* \O_Y(mB) \tensor \O_X(\ell \wt H)) = 0 \qquad \text{for $i > 0$, $m \gg 0$ and $\ell \ge 0$.} \]
By Castelnuovo--Mumford regularity~\cite[Thm.~1.8.5]{Laz04a}, for $m \gg 0$ the sheaf
\[ \sF_m := f_* \O_Y(mB) \tensor \O_X(m \wt H) \]
is globally generated and satisfies $H^1(X, \sF_m) = 0$.

\subsubsection*{Step 3: Pushing down}

Observe that $\sF_m$ is torsion-free and that its reflexive hull $\sF_m^{**}$ is isomorphic to $\O_X(m(\wt D + \wt H))$. Indeed, both sheaves are reflexive and they agree outside of $W$, which has codimension at least two in $X$.
Thus the natural map $\sF_m \to \sF_m^{**}$ yields a short exact sequence
\[ 0 \lto \sF_m \lto \O_X(m(\wt D + \wt H)) \lto \cQ_m \lto 0, \]
where $\cQ_m$ is supported on $W$. We are aiming to show that the middle term $\O_X(m(\wt D + \wt H))$ is globally generated on $U$.
So let $x \in U$ be arbitrary. Then either $x \not\in W$, whence the stalk $\cQ_{m,x}$ is zero, or $x \in W$ is isolated, hence so is $x \in \supp \cQ_m$. In either case, $\cQ_m$ is globally generated at $x$.
By Lemma~\ref{lem:ext glob gen}, also $\O_X(m(\wt D + \wt H))$ is globally generated at $x$. Since $x \in U$ was arbitrary, it follows that $\O_X(m(\wt D + \wt H))$ is globally generated on $U$.

\subsubsection*{Step 4: End of proof}

We have shown that there is an $m_0 \in \N$ such that for $m \ge m_0$, the sheaf $\O_X(m(\wt D + \wt H))$ is globally generated on $U$. Since $m(\wt D + \wt H) = m N_0 \cdot (D + H)$, this proves the claim of the theorem if we take ``$m$ sufficiently divisible'' to mean ``$m$ divisible by $m_0 N_0$''. \qed

\section{Discreteness for isolated non-\Q-Gorenstein loci}

This section is devoted to the proof of Theorem~\ref{thm:discrete isol}, repeated here for the reader's convenience.

\begin{thm}[Discreteness for isolated non-\Q-Gorenstein loci] \label{thm:discrete isol'}
Let $(X, Z)$ be a pair such that the non-\Q-Cartier locus of $K_X$ is zero-dimensional. Then $\Xi(X, Z)$ is a discrete subset of $\R$.
\end{thm}

\subsection{Auxiliary results}

We begin with a few easy observations.

\begin{lem}[Descending chains of ideals] \label{lem:desc chain}
Let $X$ be a projective variety, and let
\[ \cJ_1 \supset \cJ_2 \supset \cdots \]
be a descending chain of coherent ideal sheaves on $X$. Let $U \subset X$ be an open subset. Assume that there exists a line bundle $\sL \in \Pic X$ such that for any $k \ge 1$, the sheaf $\sL \tensor \cJ_k$ is globally generated on $U$. Then the sequence
\[ \cJ_1|_U \supset \cJ_2|_U \supset \cdots \]
stabilizes, i.e.~we have $\cJ_k|_U = \cJ_{k+1}|_U$ for $k \gg 1$.
\end{lem}

\begin{proof}
The chain of complex vector spaces
\[ H^0(X, \sL \tensor \cJ_1) \supset H^0(X, \sL \tensor \cJ_2) \supset \cdots \]
stabilizes for dimension reasons. So for $k \gg 1$ we have a diagram
\[ \xymatrix{
H^0(X, \sL \tensor \cJ_k) \tensor_\C \O_U \ar@{->>}[d] & H^0(X, \sL \tensor \cJ_{k+1}) \tensor_\C \O_U \ar@{=}[l] \ar@{->>}[d] \\
(\sL \tensor \cJ_k)|_U & (\sL \tensor \cJ_{k+1})|_U. \ar@{ il->}[l]
} \]
This implies $(\sL \tensor \cJ_k)|_U = (\sL \tensor \cJ_{k+1})|_U$, and then $\cJ_k|_U = \cJ_{k+1}|_U$.
\end{proof}

\begin{prp}[Global generation of twisted multiplier ideals] \label{prp:gg mult ideal}
Let $(X, Z)$ be a pair, where $X$ is projective, and let $c \ge 0$ be a real number. Choose a boundary $\Delta$ on $X$ and a Cartier divisor $B$ such that $\O_X(B) \tensor \cJ_Z$ is globally generated. Furthermore let $L$ be a very ample Cartier divisor such that $L - (K_X + \Delta + cB)$ is big and nef. Then for $n \ge \dim X + 1$, the sheaf
\[ \O_X(nL) \tensor \cJ((X, \Delta), cZ) \]
is globally generated.
\end{prp}

\begin{proof}
This is a combination of Nadel vanishing in the form of~\cite[Prop.~9.4.18]{Laz04b} and Castelnuovo--Mumford regularity.
\end{proof}

\subsection{Proof of Theorem~\ref{thm:discrete isol'}}

Again, for the sake of readability the proof is divided into five steps.

\subsubsection*{Step 0: Setup of notation}

We need to show that $\Xi(X, Z)$ is discrete. Arguing by contradiction, assume that $t_0 \in \R$ is an accumulation point. Then there is a sequence $(t_k) \subset \Xi(X, Z) \minus \{ t_0 \}$ converging to $t_0$. By Proposition~\ref{prp:elem}(\ref{itm:elem.dcc}), we may assume that the sequence $(t_k)$ is strictly increasing.

\subsubsection*{Step 1: Simplifying assumptions}

Cover $X$ by finitely many affine open subsets $U_i$. Then clearly
\[ \Xi(X, Z) = \bigcup_i \Xi(U_i, Z|_{U_i}). \]
For some index $i$, the set $\Xi(U_i, Z|_{U_i})$ contains a subsequence of $(t_k)$. Hence we may assume that $X = U_i \subset \A^N$ is affine. Taking the closure of $X$ in $\P^N$ and normalizing yields the following assumption.

\begin{awlog}
The variety $X$ is embedded as an open set $X \subset \wb X$ in a normal projective variety $\wb X$.
\end{awlog}

We denote by $\wb Z$ the closure of $Z$ considered as a locally closed subscheme of $\wb X$.

\subsubsection*{Step 2: Constructing compatible boundaries}

Let $D$ be an effective Weil divisor on $X$ such that $K_X - D$ is Cartier, and let $\wb D$ be its closure in $\wb X$. By our assumptions, the non-\Q-Cartier locus of $D$ is finite. Hence by Theorem~\ref{thm:glob gen isol}, there exists an ample Cartier divisor $H$ on $\wb X$ and a positive integer $m_0$ such that $\O_{\wb X}(m(H - \wb D))$ is globally generated on $X$ for all $m \in m_0 \cdot \N$.
For such $m$, define
\[ \wb\Delta_m := \frac{1}{m} M, \; \text{where $M \in |m(H - \wb D)|$ is a general element.} \]
Then by Theorem~\ref{thm:ex comp bd}, the divisor $\Delta_m := \wb\Delta_m|_X$ is an $m$-compatible boundary for the pair $(X, Z)$.
Note that the \Q-linear equivalence class of $\wb\Delta_m \sim_\Q H - \wb D$ does not depend on $m$.
Note also that $\wb\Delta_m$ need not be the closure of $\Delta_m$, as $\wb\Delta_m$ might have components contained in $\wb X \minus X$.

\subsubsection*{Step 3: Global generation}

Let $B$ be an ample Cartier divisor on $\wb X$ such that $\O_{\wb X}(B) \tensor \cJ_{\wb Z}$ is globally generated.
As we have remarked above, the numerical equivalence class of $K_{\wb X} + \wb\Delta_m$ is independent of $m$. Thus we can find a very ample Cartier divisor $L$ on $\wb X$ such that
\[ L - (K_{\wb X} + \wb\Delta_m + t_0 B) \quad \text{is big and nef for all $m \in m_0 \cdot \N$.} \]
Then also $L - (K_{\wb X} + \wb\Delta_m + t_k B)$ is big and nef for $k \ge 1$.
Fix $n \ge \dim X + 1$. By Proposition~\ref{prp:gg mult ideal},
\[ \O_{\wb X}(nL) \tensor \cJ((\wb X, \wb\Delta_m), t_k \wb Z) \quad \text{is globally generated for all $m \in m_0 \cdot \N$, $k \ge 1$.} \]

\begin{clm} \label{clm:gg}
The sheaf $\O_{\wb X}(nL) \tensor \cJ(\wb X, t_k \wb Z)$ is globally generated on $X$ for all $k \ge 1$.
\end{clm}

\begin{proof}[Proof of Claim~\ref{clm:gg}]
For $m \in m_0 \cdot \N$ sufficiently divisible, we have $\cJ(X, t_k Z) = \cJ_m(X, t_k Z)$. Fix such an $m$. Then, since $\Delta_m$ is $m$-compatible,
\[ \cJ(\wb X, t_k \wb Z)|_X = \cJ(X, t_k Z) = \cJ((X, \Delta_m), t_k Z) = \cJ((\wb X, \wb\Delta_m), t_k \wb Z)|_X. \]
We also have
\[ \O_{\wb X}(nL) \tensor \cJ((\wb X, \wb\Delta_m), t_k \wb Z) \subset \O_{\wb X}(nL) \tensor \cJ(\wb X, t_k \wb Z) \]
by~\cite[Rem.~5.3]{dFH09}. As the left-hand side sheaf is globally generated, so is the right-hand side one wherever they agree. In particular, this is the case on the open subset $X \subset \wb X$.
\end{proof}

\subsubsection*{Step 4: End of proof}

Consider the descending chain
\[ \cJ(\wb X, t_1 \wb Z) \supset \cJ(\wb X, t_2 \wb Z) \supset \cdots. \]
By Claim~\ref{clm:gg} and Lemma~\ref{lem:desc chain}, the restriction of this chain to $X$ stabilizes. This restriction is nothing but
\[ \cJ(X, t_1 Z) \supset \cJ(X, t_2 Z) \supset \cdots. \]
However, by the definition of jumping numbers we have $\cJ(X, t_k Z) \supsetneq \cJ(X, t_{k+1} Z)$ for all $k$. This is a contradiction, showing that $\Xi(X, Z)$ is discrete and thus finishing the proof. \qed

\section{Generic numerical \Q-factoriality}

In this section, we prove Theorem~\ref{thm:gen num Q-fact} from the introduction:

\begin{thm}[Generic numerical \Q-factoriality] \label{thm:gen num Q-fact'}
Let $X$ be a normal variety. Then there is a closed subset $W \subset X$ of codimension at least three such that $X \minus W$ is numerically \Q-factorial.
\end{thm}

\begin{proof}
Let $f\!: \wt X \to X$ be a log resolution of $X$, with exceptional locus $E = E_1 + \cdots + E_k$. Re-indexing, we may assume that for some number $\ell$, we have $\codim_X f(E_i) = 2$ for $1 \le i \le \ell$, while $\codim_X f(E_i) \ge 3$ for $\ell < i \le k$.
We may remove the closed set $\bigcup_{\ell < i \le k} f(E_i)$ from $X$, since it has codimension at least three. Furthermore, by generic smoothness~\cite[Ch.~III, Cor.~10.7]{Har77}, for each $1 \le i \le \ell$ the morphism $f_i := f|_{E_i}\!: E_i \to f(E_i)$ is smooth over a smooth dense open subset $V_i$ of $f(E_i)$. The union $\bigcup_{i=1}^\ell f(E_i) \minus V_i$ is again closed of codimension at least three in $X$, hence we may remove it. Put together, this yields the following additional assumption.

\begin{awlog}
The exceptional locus of $f$ is $E = E_1 + \cdots + E_\ell$.
For any $1 \le i \le \ell$, we have that $f(E_i)$ is smooth of codimension exactly two in $X$ and that the morphism $f_i\!: E_i \to f(E_i)$ is smooth.
\end{awlog}

We then have the following claim.

\begin{clm} \label{clm:fibre}
For any index $1 \le i \le \ell$, there is a natural number $n_i$ and a numerical class $\gamma_i \in N_1(\wt X / X)_\Q$ such that the following holds.
For any $x \in f(E_i)$, we have that $f_i^{-1}(x) = C_i^{(1)} \cup \cdots \cup C_i^{(n_i)}$ is a disjoint union of smooth curves, all of which are numerically equivalent to $\gamma_i$, i.e.~we have $[C_i^{(j)}] = \gamma_i \in N_1(\wt X / X)_\Q$ for all $1 \le j \le n_i$.
\end{clm}

\begin{proof}[Proof of Claim~\ref{clm:fibre}]
Fix an index $i$. The morphism $f_i\!: E_i \to f(E_i)$ is smooth of relative dimension one. Let
\[ \xymatrix{
E_i \ar[r]^{g_i} & B_i \ar[r]^{h_i} & f(E_i)
} \]
be its Stein factorization. Then $g_i$ is smooth of relative dimension one with connected fibres, and $h_i$ is finite \'etale, say of degree $n_i$.

It is then clear that for any $x \in f(E_i)$, the fibre $f_i^{-1}(x)$ is a disjoint union of $n_i$ many smooth curves $C_i^{(j)}$. Each of these curves is a (scheme-theoretic) fibre of $g_i$, hence they all represent the same class $\Gamma_i \in N_1(E_i / B_i)_\Q = N_1(E_i / f(E_i))_\Q$, independent of the point $x$. If $\gamma_i$ is the image of $\Gamma_i$ under the natural map $N_1(E_i / f(E_i))_\Q \to N_1(\wt X / X)_\Q$, then $[C_i^{(j)}] = \gamma_i$ for all indices $j$.
\end{proof}

Claim~\ref{clm:fibre} implies that $N_1(\wt X / X)_\Q$ is spanned by $\gamma_1, \dots, \gamma_\ell$. In particular, the relative Picard number of $f$ is $\rho(\wt X / X) = \ell$. Hence the dimension of $N^1(\wt X / X)_\Q$ is likewise $\ell$. We now make use of the short exact sequence of \Q-vector spaces
\[ 0 \lto \bigoplus_{i=1}^\ell \Q \cdot [E_i] \lto N^1(\wt X / X)_\Q \lto \Cln(X)_\Q \lto 0 \]
given by Proposition~\ref{prp:ses}. By what we have just proved, the map on the left is an isomorphism for dimension reasons, which yields that $\Cln(X)_\Q = 0$. This means that $X$ is numerically \Q-factorial.
\end{proof}

\section{Discreteness in dimension three}

Recall that Theorem~\ref{thm:discrete 3} states the following:

\begin{thm}[Discreteness in dimension three] \label{thm:discrete 3'}
Let $(X, Z)$ be a pair, where $X$ is a normal threefold whose locus of non-rational singularities is zero-dimensional. Then $\Xi(X, Z)$ is a discrete subset of $\R$.
\end{thm}

\begin{proof}[Proof of Theorem~\ref{thm:discrete 3'}]
By Theorem~\ref{thm:gen num Q-fact} and the assumption, there is a finite subset $W \subset X$ such that $X \minus W$ is numerically \Q-factorial and has rational singularities. It then follows from Theorem~\ref{thm:ratl} that $X \minus W$ is \Q-factorial, i.e.~every Weil divisor on $X$ is \Q-Cartier outside a finite set of points. In particular, the non-\Q-Cartier locus of $K_X$ is zero-dimensional. Now Theorem~\ref{thm:discrete isol} applies to show that $\Xi(X, Z)$ is discrete.
\end{proof}

\section{Proof of Corollaries}

\subsection*{Proof of Corollary~\ref{cor:gen discrete}}


By Theorem~\ref{thm:gen num Q-fact'}, there is an open subset $U \subset X$ such that $X \minus U$ has codimension at least three in $X$, and $U$ is numerically \Q-factorial. In particular, $K_U$ is numerically \Q-Cartier. Note that $U$ does not depend on $Z$. Let $\pi\!: \wt U \to U$ be a log resolution of $Z|_U \subset U$, so that $\cJ_{Z|_U} \cdot \O_{\wt U} = \O_{\wt U}(-D)$ for some effective Cartier divisor $D$ on $\wt U$. We may write
\[ K_{\wt U/U}^- = \sum_{i=1}^k a_i E_i
\quad \text{and} \quad
D = \sum_{i=1}^k r_i E_i \]
for some $k \in \N$ and suitable $a_i \in \Q$, $r_i \in \N_0$ and prime divisors $E_i$ on $\wt U$.

By Theorem~\ref{thm:mult ideals num Q-Gor}, for any $t > 0$ we have
\[ \cJ(U, tZ|_U) = \pi_* \O_{\wt U} \left( \rp K_{\wt U/U}^- - tD. \right)
 = \pi_* \O_{\wt U} \left( \sum_{i=1}^k \rp a_i - t \cdot r_i. E_i \right). \]
This implies that if $t$ is a jumping number of $(U, Z|_U)$, then $a_i - t \cdot r_i$ is a negative integer for some index $1 \le i \le k$ such that $r_i \ne 0$. Hence
\[ \Xi(U, Z|_U) \subset \left\{ \frac{ a_i + m }{ r_i } \;\;\middle|\;\; \text{$\exists \; 1 \le i \le k$ such that $r_i \ne 0$, and $m \in \N$} \right\}. \]
The set on the right-hand side is clearly discrete and consists of rational numbers only. So $\Xi(U, Z|_U)$ enjoys the same properties. This ends the proof.

\subsection*{Proof of Corollary~\ref{cor:gen test}}

Immediate from Theorem~\ref{thm:gen num Q-fact} and~\cite[Thm.~1]{dFDTT14}.



\providecommand{\bysame}{\leavevmode\hbox to3em{\hrulefill}\thinspace}
\providecommand{\MR}{\relax\ifhmode\unskip\space\fi MR}
\providecommand{\MRhref}[2]{%
  \href{http://www.ams.org/mathscinet-getitem?mr=#1}{#2}
}
\providecommand{\href}[2]{#2}

\end{document}